\documentclass[11pt,reqno]{article}

\usepackage{amsmath, amssymb, amsthm} 
\usepackage{empheq}
\usepackage{mathpazo}
\usepackage{nameref}
\usepackage[margin=3cm]{geometry}
\usepackage[titletoc]{appendix}

\usepackage{caption,subcaption}
\usepackage[doc]{optional}
\usepackage{graphicx, xcolor, soul}
\usepackage{float}
\restylefloat{table*}
\usepackage{booktabs} 
\definecolor{labelkey}{rgb}{0,0.08,0.45}
\definecolor{refkey}{rgb}{0,0.6,0.0}
\definecolor{Brown}{rgb}{0.45,0.0,0.05}
\definecolor{lime}{rgb}{0.00,0.8,0.0}
\definecolor{lblue}{rgb}{0.5,0.5,0.99}

\usepackage{hyperref}

\usepackage{sidecap}
\definecolor{myblue}{rgb}{.9, .9, 1} 
  \newcommand*\mybluebox[1]{%
    \colorbox{myblue}{\hspace{1em}#1\hspace{1em}}}

\newcommand{\sepp}{\setlength{\itemsep}{-2pt}}

\newcommand{\menge}[2]{\left\{{#1}~\big |~{#2}\right\}}

\newcommand{\norm}[1]{\left\| {#1} \right\|}

\newcommand{\scal}[2]{\left\langle {#1},{#2} \right\rangle}

\newcommand{\NN}{\ensuremath{\mathbb N}}
\newcommand{\nnn}{\ensuremath{{n\in{\mathbb N}}}}

\newcommand{\RR}{\ensuremath{\mathbb R}}
\newcommand{\RP}{\ensuremath{\mathbb{R}_+}}
\newcommand{\RPP}{\ensuremath{\mathbb{R}_{++}}}

\newcommand{\inte}{\ensuremath{\operatorname{int}}}

\newcommand{\aff}{\ensuremath{\operatorname{aff}}}

\newcommand{\cran}{\ensuremath{\overline{\operatorname{ran}}\,}}
\newcommand{\ran}{\ensuremath{\operatorname{ran}}}

\newcommand{\Fix}{\ensuremath{\operatorname{Fix}}}
\newcommand{\Id}{\ensuremath{\operatorname{Id}}}

\newcommand{\sgn}{\ensuremath{\operatorname{sgn}}}

\providecommand{\innp}[1]{\langle#1\rangle}

\usepackage[capitalize]{cleveref}
\crefname{equation}{}{equations}
\crefname{chapter}{Appendix}{chapters}
\crefname{item}{}{items}
\crefname{figure}{Figure}{figures}
\makeatletter
\def\namedlabel#1#2{\begingroup
   \def\@currentlabel{#2}%
   \label{#1}\endgroup
}
\makeatother

\makeatletter
\def\th@plain{%
  \thm@notefont{}
  \itshape 
}
\def\th@definition{%
  \thm@notefont{}
  \normalfont 
}
\makeatother

\newtheorem{theorem}{Theorem}[section]
\newtheorem{lemma}[theorem]{Lemma}
\newtheorem{corollary}[theorem]{Corollary}
\newtheorem{proposition}[theorem]{Proposition}
\newtheorem{fact}[theorem]{Fact}
\theoremstyle{definition}

\theoremstyle{definition}
\newtheorem{example}[theorem]{Example}
\theoremstyle{definition}
\newtheorem{remark}[theorem]{Remark}

\allowdisplaybreaks 

\begin{document}

\title{On Fej\'er monotone sequences\\ and nonexpansive
mappings}

\author{
Heinz H.\ Bauschke\thanks{
Mathematics, University of British Columbia, Kelowna, B.C.\ V1V~1V7, Canada. 
E-mail: \texttt{heinz.bauschke@ubc.ca}.},~
Minh N.\ Dao\thanks{
Mathematics, University of British Columbia, Kelowna, B.C.\ V1V~1V7, Canada,
and
Department of Mathematics and Informatics,
Hanoi National University of Education, 136 Xuan Thuy, Hanoi, Vietnam.
E-mail: \texttt{minhdn@hnue.edu.vn}.}, 
~and Walaa M.\ Moursi\thanks{
Mathematics, University of British Columbia, Kelowna, B.C.\ V1V~1V7, Canada,
and 
Mansoura University, Faculty of Science, Mathematics Department, 
Mansoura 35516, Egypt. 
E-mail: \texttt{walaa.moursi@ubc.ca}.}
}

\date{July 20, 2015}

\maketitle

\begin{abstract}
\noindent
The notion of Fej\'er monotonicity has proven to be a fruitful concept in
fixed point theory and optimization. In this paper, we present
new conditions sufficient for convergence of Fej\'er monotone
sequences and we also provide applications to the study of
nonexpansive mappings. Various examples illustrate our results.
\end{abstract}
{\small

}

\section{Introduction}
We assume throughout the paper that
\begin{empheq}[box=\mybluebox]{equation}
X \text{~~is a real Hilbert space}
\end{empheq} 
with inner product $\scal{\cdot}{\cdot}$ and induced norm $\|\cdot\|$.
Let $C$ be a nonempty closed convex subset of $X$. 
A sequence $(x_n)_\nnn$ in $X$ is called \emph{Fej\'er monotone} 
(see, e.g., \cite{Comb01:1}, \cite{Comb01:2} and \cite{BC01})
with respect to $C$ if
\begin{equation}
(\forall c \in C)(\forall\nnn)\quad \norm{x_{n+1} -c} \leq \norm{x_n -c}.
\end{equation}
In other words, each point in a Fej\'er monotone sequence is not further from
any point in $C$ than its predecessor. 
This property has known to be an efficient tool 
to analyze various iterative algorithms in convex optimization.

\emph{The goal of this paper is to present some new conditions
sufficient for convergence of Fej\'er monotone sequences. We also
provide applications to the study of nonexpansive mappings.}

The paper is organized as follows.
In Section~2, we deal with Fej\'er monotonicity.
Section~3 is devoted to applications in fixed point theory.
Section~4 concludes the paper with a list of open problems.

The notation we employ is standard and follows, e.g.,
\cite{BC11}. 

\section{Fej\'er monotonicity}
We start by recalling some pleasant properties of Fej\'er
monotone sequences.

\begin{fact}
\label{f:basic}
Let $(x_n)_\nnn$ be a sequence in $X$ that is
is Fej\'er monotone with respect to a nonempty closed convex
subset $C$ of $X$. 
Then the following hold:
\vspace{-0.4cm}
\begin{enumerate}
\item 
\label{f:basic_bounded}
The sequence $(x_n)_\nnn$ is bounded.
\item
\label{f:basic_norm}
For every $c\in C$, the sequence $(\norm{x_n -c})_\nnn$ converges.
\item
\label{f:basic_norm+}
The set of strong cluster points of $(x_n)_\nnn$ lies in a sphere of $X$. 
\item
\label{f:basic_shadow}
The ``shadow sequence'' 
$(P_C x_n)_\nnn$ converges strongly to a point in $C$.
\item
\label{f:basic_int}
If $\inte C \neq\varnothing$, then $(x_n)_\nnn$ converges
strongly to a point in $X$.
\item
\label{f:basic_affine}
If $C$ is a closed affine subspace of $X$, then 
$(\forall\nnn)$ $P_Cx_n=P_Cx_0$. 
\item
\label{f:basic_weakstrong}
Every weak cluster point of $(x_n)_\nnn$ that belongs to $C$ must
be 
$\lim_{n\to\infty} P_Cx_n$. 
\item
\label{f:basic_weakchar}
The sequence $(x_n)_\nnn$ converges weakly to some point in $C$ if and only if all
weak cluster points of $(x_n)_\nnn$ lie in $C$. 
\item
\label{f:basic_weakchar2}
If all weak cluster points of $(x_n)_\nnn$ lie in $C$,
then $(x_n)_\nnn$ converges weakly to $\lim_{n\to\infty}P_Cx_n$. 
\end{enumerate}
\end{fact}
\begin{proof}
\ref{f:basic_bounded}\&\ref{f:basic_norm}: \cite[Proposition~5.4]{BC11}. 
\ref{f:basic_norm+}: Clear from \ref{f:basic_norm}.  
\ref{f:basic_shadow}: \cite[Proposition~5.7]{BC11}.
\ref{f:basic_int}: \cite[Proposition~5.10]{BC11}.
\ref{f:basic_affine}: \cite[Proposition~5.9(i)]{BC11}. 
\ref{f:basic_weakstrong}: This follows from \cite[Corollary~5.11]{BC11}. 
\ref{f:basic_weakchar}: This follows from \cite[Theorem~5.5]{BC11}. 
\ref{f:basic_weakchar2}: Combine \ref{f:basic_weakchar} with
\ref{f:basic_weakstrong}. 
\end{proof}

The following result was first presented in 
\cite[Theorem~6.2.2(ii)]{Bau96}; for completeness,
we include its short proof. 

\begin{lemma}\label{f:basic_cluster}
Let $(x_n)_\nnn$ be a sequence in $X$ that is
is Fej\'er monotone with respect to a nonempty closed convex
subset $C$ of $X$. 
Let $w_1$ and $w_2$ be weak cluster points of $(x_n)_\nnn$.
Then $w_1 -w_2 \in (C -C)^\perp$.
\end{lemma}
\begin{proof}
Let $(c_1,c_2)\in C\times C$. Using
\cref{f:basic}\ref{f:basic_norm}, set 
$L_i:=\lim_{n\to \infty}\norm{x_n-c_i}$, for $i\in \{1,2\}$.
Note that 
\begin{equation}
\label{eq:2:sub}
\norm{x_n-c_1}^2=\norm{x_n-c_2}^2
+\norm{c_1-c_2}^2+2\innp{x_2-c_2,c_2-c_1}.
\end{equation}
Now suppose that $x_{k_n}\rightharpoonup w_1$
and $x_{l_n}\rightharpoonup w_2$.
Taking the limit in \cref{eq:2:sub} along the two subsequences 
$(k_n)_\nnn$ and $(l_n)_\nnn$ yields 
$L_1=L_2+\norm{c_2-c_1}^2+2\innp{w_1-w_2,c_2-c_1}$
and 
$L_1=L_2+\norm{c_2-c_1}^2+2\innp{w_2-w_2,c_2-c_1}$.
Subtracting the last two equations yields
$2\scal{c_2-c_1}{w_1-w_2}=0$. 
\end{proof}

We are now ready for our first result 
which can be seen as a finite-dimensional variant of 
\cite[Lemma~2.1]{BDM15} (where $A$ is a closed linear subspace)
and 
\cref{f:basic}\ref{f:basic_weakchar} (where $A=X$). 

\begin{proposition}
\label{p:shadow}
Suppose that $X$ is finite-dimensional,
let $(x_n)_\nnn$ be a sequence in $X$ that is
Fej\'er monotone with respect to a nonempty closed convex
subset $C$ of $X$, and let $A$ be a closed convex subset of $X$
such that $C\subseteq A$.
If all cluster points of $(P_Ax_n)_\nnn$ lie in $C$,
then $(P_A x_n)_\nnn$ converges; in fact, 
\begin{equation}
\label{e:level3done}
\lim_{n\to \infty}P_A x_n
=\lim_{n\to \infty}P_C x_n.
\end{equation}
\end{proposition}
\begin{proof}
Set $c^* := \lim_{n\to\infty} P_Cx_n$ (see
\cref{f:basic}\ref{f:basic_shadow}). 
By \cref{f:basic}\ref{f:basic_bounded}, $(x_n)_\nnn$ is bounded,
hence so is $(P_Ax_n)_\nnn$ because $P_A$ is nonexpansive.
Now assume that all cluster points of $(P_Ax_n)_\nnn$ 
lie in $C$. 
Let $c$ be an arbitrary cluster point of $(P_Ax_n)_\nnn$. 
Then there exist a subsequence $(x_{k_n})_\nnn$ of $(x_n)_\nnn$
and a point $x\in X$ such that 
$x_{k_n} \to x$ and $P_A x_{k_n} \to P_A x =c \in C$.
It follows that $c^* \leftarrow P_Cx_{k_n} \to P_Cc = c$.
Hence $c=c^*$ and the result follows. 
\end{proof}

Our second result decouples Fej\'er monotonicity into two
properties in the case when
the underlying set can be written as the sum of a set and a cone.

\begin{proposition}
\label{p:sum}
Let $(x_n)_\nnn$ be a sequence in $X$, 
let $E$ be a nonempty subset of $X$ 
and let $K$ be a nonempty convex cone of $X$.
Then the following are equivalent:
\vspace{-0.4cm}
\begin{enumerate}
\item 
\label{p:sum_CK}
$(x_n)_\nnn$ is Fej\'er monotone with respect to $E +K$.
\item
\label{p:sum_C}
$(x_n)_\nnn$ is Fej\'er monotone with respect to $E$ and
$(\forall\nnn)$ $x_{n+1} \in x_n +K^\oplus$, where
$K^\oplus := \menge{u\in X}{\inf\scal{u}{K}\geq 0}$. 
\end{enumerate}
\end{proposition}
\begin{proof}
Set 
\begin{equation}
\label{e:0718a}
(\forall x \in X)(\forall\nnn)\quad \Delta_n(x) :=\norm{x_n -x}^2 -\norm{x_{n+1} -x}^2.
\end{equation}
Then for every $e \in E$ and $k \in K$, we have
\begin{subequations}
\label{e:diff}
\begin{align}
\Delta_n(e +k) &=\norm{x_n -e}^2 +\norm{k}^2 -2\scal{x_n -e}{k} 
-\big(\norm{x_{n+1} -e}^2 +\norm{k}^2 -2\scal{x_{n+1} -e}{k}\big) \\ 
&=\Delta_n(e) +2\scal{x_{n+1} -x_n}{k}. 
\end{align}
\end{subequations}
Assume first that \ref{p:sum_CK} holds. 
Then
$(x_n)_\nnn$ is Fej\'er monotone with respect to $E$ because 
$E \subseteq E +K$.
Let $(e,k)\in E\times K$ and $\nnn$.
Using \eqref{e:0718a},
\begin{equation}
0 \leq \Delta_n(e +k) 
=\Delta_n(e) +2\scal{x_{n+1} -x_n}{k}. 
\end{equation}
Since $K$ is a cone, this shows that
$2\inf\scal{x_{n+1}-x_n}{\RPP k} \geq -\Delta_n(e)>-\infty$.
Hence $\scal{x_{n+1}-x_n}{k}\geq 0$.
It follows that $x_{n+1}-x_n\in K^\oplus$. 
Conversely, if \ref{p:sum_C} holds,
then 
\cref{e:diff} immediately yields \ref{p:sum_CK}.
\end{proof}

The following consequence of \cref{p:sum} shows that \cref{p:sum}
is a generalization of \cref{f:basic}\ref{f:basic_affine}. 

\begin{corollary}
Let $(x_n)_\nnn$ be a sequence in $X$, and let $C$ be a closed
affine subspace of $X$, say $C = c + Y$, where $Y$ is a closed
linear subspace of $X$.
Then $(x_n)_\nnn$ is Fej\'er monotone with respect to $Y$ if and
only if $(\forall\nnn)$ $\|x_{n+1}-c\|\leq \|x_n-c\|$ and 
$x_{n+1}\in x_n+Y^\perp$, in which case
$(P_Cx_n)_\nnn$ is a constant sequence. 
\end{corollary}

We continue with the following lemma, 
which is a slight generalization of a theorem of Ostrowski 
(see \cite[Theorem 26.1]{Ost73}) whose
proof we follow.

\begin{lemma}
\label{l:connect}
Let $(Y, d)$ be a metric space, 
and let $(x_n)_\nnn$ be a sequence in a compact subset $C$ of $Y$
such that $d(x_n, x_{n+1}) \to 0$.
Then the set of cluster points of $(x_n)_\nnn$ 
is a compact connected subset of $C$.
\end{lemma}
\begin{proof}
Denote the set of cluster points of $(x_n)_\nnn$ by $S$
and assume to the contrary that 
$S =A\cup B$ where $A$ and $B$ are nonempty closed 
subsets of $X$ and $A\cap B =\varnothing$.
Then 
\begin{equation}
\label{e:ost:cont}
\delta := \inf_{(a,b)\in A\times B} d(a, b) > 0. 
\end{equation}
By assumption on $(x_n)_\nnn$, there exists 
$n_0\in \NN$  such that $(\forall n \geq n_0)$ 
$d(x_n, x_{n+1}) \leq \delta/3$.
Let $a \in A$. Then there exists $m >n_0$  such that 
$d(x_m, a) <\delta/3$. 
Because $(x_n)_{n>m}$ has a cluster point in $B$, 
there exists a smallest integer $k>m$ such that 
$d(x_k, B) < 2\delta/3$. 
Then $d(x_{k-1}, B) \geq 2\delta/3$ and
hence $d(x_k,B) \geq d(x_{k-1},B)-d(x_{k-1},x_k) \geq
2\delta/3-\delta/3 = \delta/3$.
Thus $\delta/3 \leq d(x_k,B) < 2\delta/3$.
Repeating this argument yields a subsequence 
$(x_{k_n})_\nnn$ of $(x_n)_\nnn$ such that 
$(\forall\nnn)$
$\delta/3 \leq d(x_{k_n},B) < 2\delta/3$. 
Let $x$ be a cluster point of $(x_{k_n})_\nnn$.
It follows that 
\begin{equation}
\delta/3 \leq d(x,B) \leq 2\delta/3
\end{equation}
Obviously, $x\notin B$. 
Hence $x\in A$, and therefore (recall \cref{e:ost:cont}) 
$\delta \leq \delta(x,B) \leq 2\delta/3 < \delta$, which is
absurd. 
\end{proof}

An immediate consequence of Lemma~\ref{l:connect} is the
classical Ostrowski result. 

\begin{corollary}[Ostrowski]
\label{c:connect}
Suppose that $X$ is finite-dimensional and 
let $(x_n)_\nnn$ be a bounded sequence in $X$ such that 
$(x_n)_\nnn$ is \emph{asymptotically regular}, i.e., 
$x_n -x_{n+1} \to 0$.
Then the set of cluster points of $(x_n)_\nnn$ is compact and connected.
\end{corollary}

We are now in position to prove the following key result
which can be seen as a variant of
\cref{f:basic}\ref{f:basic_int}.

\begin{theorem}[a new sufficient condition for convergence]
\label{p:codim1}
Suppose that $X$ is finite-dimensional 
and that $C$ is a nonempty closed convex subset of $X$ 
of co-dimension 1, i.e., 
\begin{equation}
\label{e:codim1}
\operatorname{codim}C :=\operatorname{codim}(\aff C -\aff C) =1.
\end{equation}
Let $(x_n)_\nnn$ be a sequence that is Fej\'er monotone with
respect to $C$ and asymptotically regular, i.e., 
$x_n-x_{n+1}\to 0$.
Then $(x_n)_\nnn$ is actually convergent. 
\end{theorem}
\begin{proof}
By \cref{f:basic}\ref{f:basic_bounded}, $(x_n)_\nnn$ is bounded.
Denote by $S$ the set of cluster points of $(x_n)_\nnn$.
Since $x_n -x_{n+1} \to 0$, \cref{c:connect} 
implies that $S$ is connected.
Moreover, $S$ lies in a sphere of $X$ 
due to \cref{f:basic}\ref{f:basic_norm+}.
On the other hand, by combining \cref{f:basic_cluster} 
and \cref{e:codim1}, 
$S$ lies in a line of $X$.
Altogether $S$ is a connected subset of 
a sphere that lies on a line. 
We deduce that $S$ is a singleton. 
\end{proof}

We conclude with two examples illustrating
that the assumptions on asymptotic regularity and co-dimension 1
are important.

\begin{example}
Suppose that $X=\RR^2$, set $C=\{0\}\times \RR$, and 
$(\forall \nnn)$ $x_n=((-1)^n,0)$.
Then $\operatorname{codim}C =1$  and $(\forall c\in C)$
$(\forall \nnn)$ $\norm{x_n-c}=\norm{x_{n+1}-c}$,
hence
$(x_n)_\nnn$ is 
Fej\'{e}r monotone with respect to $C$.
However, $(x_n)_\nnn$ does not converge.
This does not contradict \cref{p:codim1} because
$\norm{x_n-x_{n+1}}=2\not \to 0$.
\end{example}

\begin{example}
Suppose that $X=\RR^2$, set $C=\{(0,0)\}\subseteq X$,
and $(\forall\nnn)$
$\theta_n = \sum_{k=1}^n (1/k)$
and 
$x_n = \cos(\theta_n)(1,0) + \sin(\theta_n)(0,1)$.
Then $(x_n)_\nnn$ is asymptotically regular and Fej\'er monotone
with respect to $C$. 
However, the set of cluster points of $(x_n)_\nnn$ is the unit
sphere because the harmonic series diverges. 
Again, this does not contradict
\cref{p:codim1} because
$\operatorname{codim}C =2\neq 1$. 
\end{example}

\section{Asymptotic behaviour of nonexpansive mappings}

From now on, we assume that
\begin{empheq}[box=\mybluebox]{equation}
T\colon X\to X \;\;\text{is nonexpansive.}
\end{empheq} 
Let $x$ and $y$ be in $X$.
It is clear that
$(\forall\nnn)$ $\|T^{n+1}x-T^{n+1}y\|\leq\|T^nx-T^ny\|$ is
bounded. The following question is thus extremely natural:
\begin{equation}
\label{e:question}
\text{Under which conditions on $T$ must $(T^nx-T^ny)_\nnn$
always converge weakly?}
\end{equation}
We first note that \cref{e:question} will impose some restriction
on $T$: 
\begin{example}
\label{ex:-IdonR}
Suppose that $X=\RR$, that $T=-\Id$, that $x\neq 0$ and
that $y=0$. Then the sequence 
$(T^nx-T^ny)_\nnn = ((-1)^nx)_\nnn$ is not convergent. 
\end{example}

The following two results are well known.

\begin{fact}
\label{f:veryknown}
{\rm (See, e.g., \cite[Corollary~6]{Pazy}.)}
Exactly one of the following holds:
\vspace{-0.4cm}
\begin{enumerate}
\item 
$\Fix(T) = \varnothing$ and $(\forall x\in X)$
$\|T^nx\|\to\infty$.
\item
$\Fix(T) \neq \varnothing$ and $(\forall x\in X)$ 
$(T^nx)_\nnn$ is bounded.
\end{enumerate}
\end{fact}

\smallskip

\begin{fact}
\label{f:prettyknown}
{\rm (See, e.g., \cite[Theorem~1.2]{BBR78}.)}
Suppose that $\Fix(T)\neq\varnothing$ and let $x\in X$.
Then $(T^nx)_\nnn$ is weakly convergent if and only if 
$T^nx-T^{n+1}x\rightharpoonup 0$; if this is the case, then $(T^nx)_\nnn$
converges weakly to a point in $\Fix(T)$.
\end{fact}

To make further progress, let us recall that 
$\cran(\Id-T)$ is a nonempty closed convex set, which makes the
vector 
\begin{empheq}[box=\mybluebox]{equation}
\label{e:v}
v :=P_{\overline{\ran}(\Id -T)}0
\end{empheq} 
well defined (see \cite{BR77}, \cite{BBR78} and \cite{Pazy}), and
which gives rise to the generalized (possibly empty) fixed point set
\begin{equation}
\Fix(v+T) = \menge{x\in X}{x = v+Tx}.
\end{equation}

We now recall the following helpful fact.

\begin{fact}
\label{f:gap}
{\rm (See \cite[Proposition~2.4]{BM15}.)}
Suppose that $\Fix(v+T)\neq\varnothing$.
Then the following hold:
\vspace{-0.4cm}
\begin{enumerate}
\item
\label{f:gap_ray}
$\Fix(v+T) -\RP v \subseteq \Fix(v +T)$. 
\item
\label{f:gap_TFix}
$(\forall y\in \Fix(v+T))(\forall\nnn)$ $T^n y =y -nv$.
\item
\label{f:gap_Fejer}
For every $x\in X$, the sequence $(T^n x +nv)_\nnn$ 
is Fej\'er monotone with respect to $\Fix(v +T)$.
\end{enumerate}
\end{fact}

\begin{remark}
\label{r:gap}
Suppose that $\Fix(v+T)\neq\varnothing$. 
Then 
\begin{subequations}
\begin{equation}
\label{e:twovec}
(\forall x\in X)(\forall y\in X)\quad
(T^nx-T^ny)_\nnn \;\;\text{is weakly convergent}
\end{equation}
if and only if 
\begin{equation}
\label{e:onevec}
(\forall x\in X)\quad
(T^nx+nv)_\nnn \;\;\text{is weakly convergent}.
\end{equation}
\end{subequations}
Indeed, if \eqref{e:twovec} holds, then
\eqref{e:onevec} follows by choosing $y\in\Fix(v+T)$ and
recalling \cref{f:gap}\ref{f:gap_TFix}.
Conversely, assume that \eqref{e:onevec} holds.
Then $(T^nx+nv)_\nnn$ and $(T^ny+nv)_\nnn$ are weakly convergent,
and so is their difference which yields \eqref{e:twovec}.
\end{remark}

We can now give a mild sufficient condition for
\eqref{e:question}:

\begin{theorem}
\label{P:dim:1}
Suppose that $X=\RR$, that $v\neq 0$, and that $\Fix(v+T)\neq\varnothing$. 
Then the sequence $(T^n x +nv)_\nnn$ is convergent. 
\end{theorem}
\begin{proof}
By \cref{f:gap}\ref{f:gap_ray}\&\ref{f:gap_Fejer},
the sequence $(T^n x +nv)_\nnn$ is Fej\'er monotone with respect to 
$C := \Fix(v +T)$, and $C$ contains a ray. Therefore, 
$\inte C\neq\varnothing$ and 
\cref{f:basic}\ref{f:basic_int} yields the convergence of 
$(T^n x +nv)_\nnn$.
\end{proof}

\begin{remark}
Example~\ref{ex:-IdonR} shows that the assumption that $v\neq 0$
in \cref{P:dim:1} is important.
\end{remark}

\begin{theorem}
\label{p:Tny:affine}
Suppose that $T$ is affine, say $T\colon x\to Lx+b$,
where $L$ is linear and nonexpansive, and $b\in X$.
Suppose furthermore that $L$ is asymptotically
regular\footnote{Recall that $T$ is asymptotically regular at $x$
if $T^nx-T^{n+1}x\to 0$ and that $T$ is asymptotically regular if
it is asymptotically regular at every point.},
and let $x$ and $y$ be points in $X$.
Then 
\begin{equation}
T^n x -T^n y =L^n(x -y) \to P_{\Fix(L)}(x -y).
\end{equation}
\end{theorem}
\begin{proof}
Using \cite[Theorem~3.2(ii)]{BM15}, we have 
$(\forall\nnn)$ $T^n x -T^n y =L^n x - L^n y =L^n(x -y)$.
The asymptotic regularity assumption yields
$L^n(x -y) -L^{n+1}(x -y) \to 0$. 
Using \cite[Proposition~5.27]{BC11}, we see that altogether
$T^n x -T^n y =L^n(x -y) \to P_{\Fix(L)}(x -y)$ 
\end{proof}

To make further progress we impose now additional assumptions on
$T$. 
Recall that our nonexpansive $T$ is \emph{averaged} 
if 
there exist a nonexpansive operator
$R:X\to X$ and a constant $\alpha\in \left]0,1\right[$
such that $T=(1-\alpha)\Id+\alpha R$; equivalently, 
(see, e.g., \cite[Proposition~4.25]{BC11})
\begin{equation}
\label{e:averaged}
(\forall x \in X)(\forall y \in X)\quad \norm{Tx -Ty}^2 
+\tfrac{1-\alpha}{\alpha}\norm{(\Id -T)x -(\Id -T)y}^2 \leq \norm{x -y}^2.
\end{equation} 
If $\alpha=1/2$, then $T$ is said to be \emph{firmly
nonexpansive}. 
Averaged operators have proven to be a useful class in fixed
point theory and optimization; see \cite{BBR78} and
\cite{Comb04}. 

The following result yields a generalized asymptotic regularity
for averaged nonexpansive operators.

\begin{lemma}\label{f:gap_asymp}
Suppose that $T$ is averaged and that $\Fix(v+T)\neq\varnothing$. 
Then for every $x\in X$, 
$T^n x-T^{n+1}x\to v$; equivalently, $(T^n x+nv)_\nnn$
is asymptotically regular. 
\end{lemma}
\begin{proof}
Let $x\in X$ and $y\in \Fix (v+T)$.
Since $T$ is averaged, 
it follows from \cref{e:averaged}  
and 
\cref{f:gap}\ref{f:gap_TFix} that 
there exists $\alpha\in \left]0,1\right[$ such that 
\begin{equation}
(\forall\nnn) \quad
\|T^{n+1} x-T^{n+1}y\|^2\leq \|T^{n} x-T^{n}y\|^2
-\tfrac{1-\alpha}{\alpha}\|{T^n x-T^{n+1}x-v}\|^2.
\end{equation}
Telescoping yields $\sum_{n=0}^{\infty}\norm{T^n x-T^{n+1}x-v}^2<+\infty$
and consequently $T^n x-T^{n+1}x\to v$. 
\end{proof}

Amazingly, on the real line, averagedness is a sufficient
condition for \eqref{e:question}:

\begin{theorem}
\label{p:Tny:R}
Suppose that $X=\RR$ and that $T$ is averaged. 
Let $x$ and $y$ be in $\RR$. 
Then the sequence $(T^n x -T^n y)_\nnn$ is convergent. 
\end{theorem}
\begin{proof}
Set $(\forall\nnn)$ $a_n :=T^n x - T^n y$. 
We must show that $(a_n)_\nnn$ is convergent. 
From \cref{e:averaged}, there exists 
$\alpha\in \left]0,1\right[$ such that
\begin{equation}
\label{e:averagedR}
(\forall\nnn)\quad a_{n+1}^2 +\tfrac{1-\alpha}{\alpha}(a_n -a_{n+1})^2 \leq a_n^2.
\end{equation}
Set $\beta := 1-2\alpha$ and note that $0\leq|\beta|<1$. 
By viewing \cref{e:averagedR} as a quadratic inequality in
$a_{n+1}$, we learn that 
\begin{equation}
(\forall\nnn)\quad
|a_{n+1}|\leq |a_n|\;\;
\text{and $a_{n+1}$ lies between $a_n$ and $\beta a_n$.}
\end{equation}
If some $a_{n_0}=0$, then $a_n\to 0$ and we are done.
So assume that $a_n\neq 0$ for every $\nnn$.
If $(a_n)_\nnn$ changes sign only finitely many times, then
$(a_n)_\nnn$ is eventually always positive or negative.
Since $(|a_n|)_\nnn$ is decreasing, we deduce that $(a_n)_\nnn$
is convergent. 
Finally, we assume that $(a_n)_\nnn$ changes signs frequently. 
If $n\in\NN$ and $\sgn(a_{n+1}) = -\sgn(a_{n})$, then 
$|a_{n+1}|\leq |\beta||a_n|$; since this occurs 
infinitely many times, it follows that $a_n\to 0$. 
\end{proof}

\begin{theorem}
\label{thm:dim}
Suppose that $X$ is finite-dimensional, that
$T$ is averaged, that $\Fix(v+T)\neq\varnothing$, and that
that $\operatorname{codim} \Fix (v+T)\leq 1$.
Then for every $(x,y)\in X\times X$, the
sequence $(T^nx-T^ny)_\nnn$ is convergent. 
\end{theorem}
\begin{proof}
In view of \cref{r:gap}, we let $x\in X$ and
must show that $(T^nx+nv)_\nnn$ is convergent.
Set $C := \Fix(v+T)$ and 
$(\forall\nnn)$ $x_n := T^nx+nv$.
By \cref{f:gap}\ref{f:gap_Fejer}, 
$(x_n)_\nnn$ is Fej\'er monotone with respect to $C$.
Suppose first that $\operatorname{codim} C = 0$.
Then $\inte C\neq\varnothing$ and we are done by
\cref{f:basic}\ref{f:basic_int}.
Now assume that 
$\operatorname{codim} C = 1$.
By \cref{f:gap_asymp}, $(x_n)_\nnn$ is asymptotically regular.
Altogether, by \cref{p:codim1},
$(x_n)_\nnn$ is convergent. 
\end{proof}

\begin{corollary}
\label{P:dim:2}
Let $x\in X$.  
Suppose that $X=\RR^2$, that $T$ is 
averaged, that $v\neq 0$, and that 
$\Fix(v+T)\neq\varnothing$.
Then for every $(x,y)\in X\times X$, 
the sequence $(T^n x -T^ny)_\nnn$ is convergent. 
\end{corollary}
\begin{proof}
Because $v\neq 0$, \cref{f:gap}\ref{f:gap_ray} implies that 
$\dim \Fix(v+T) \geq 1$, i.e., 
$\operatorname{codim}\Fix(v+T) \leq \dim(X)-1 = 1$.
The result now follows from 
\cref{thm:dim}. 
\end{proof}

\section{Open problems}

We now present a list of open problems that
may be easier than the general question \eqref{e:question}. 
Let $x$ and $y$ be in $X$. 
\vspace{-0.4cm}
\begin{itemize}
\item[\textbf{P1:}] 
Suppose that $X=\RR$, $v=0$ but $\Fix(T) = \varnothing$.
Is $(T^nx-T^ny)_\nnn$ convergent?
\item[\textbf{P2:}] 
Suppose that $X=\RR$, $v\neq 0$ but $\Fix(v+T) = \varnothing$.
Is $(T^nx-T^ny)_\nnn$ convergent?
\item[\textbf{P3:}] 
Does \cref{P:dim:2} remain true if $\dim(X)\geq 3$?
\item[\textbf{P4:}] 
What can be said for \eqref{e:question}
if we replace ``weakly'' by ``strongly''?
\end{itemize}

Let us conclude with an example which numerically illustrates that the
answer to \textbf{P3} may be positive. 

\begin{example}
\label{ex:conj:R3}
Suppose that $X=\RR^3$ and let $A$
and $B$ be two closed balls in $X$. 
Set\footnote{For a nonempty closed convex subset $C$ of
$X $ we use $N_C$, $P_C$ and $R_C:=2P_C-\Id$ 
to denote the normal cone operator , the projector and the reflector associated with $C$,
respectively.} 
$T=\tfrac{1}{2}(\Id+R_BR_A)$.
Then $T$ is firmly nonexpansive and hence averaged.
(In fact, $T$ is the \emph{Douglas--Rachford operator} 
\cite{LM79} associated with the sets $A$ and $B$.)
It follows from \cite[Theorem~3.5]{BCL04} that $A\cap (B+v)+N_{\overline{A-B}}v\subseteq \Fix(v+T)\subseteq v+ A\cap (B+v)+N_{\overline{A-B}}v$.
Furthermore, \cite[Example~5.7]{BHM15} implies that 
$N_{\overline{A-B}}v$ is a ray, hence $\Fix(v+T)$ is ray and therefore
$\operatorname{dim} \Fix(v+T)=1$ and so 
$\operatorname{codim}\Fix(v+T)=2$. 
Even though \cref{thm:dim} is not applicable here, we still conjecture
that $(T^n x+nv)_\nnn$ converges (see Figure \ref{fig_2_balls} below).
\end{example}
\begin{figure}[h]
\begin{center}
\includegraphics[scale=0.54]{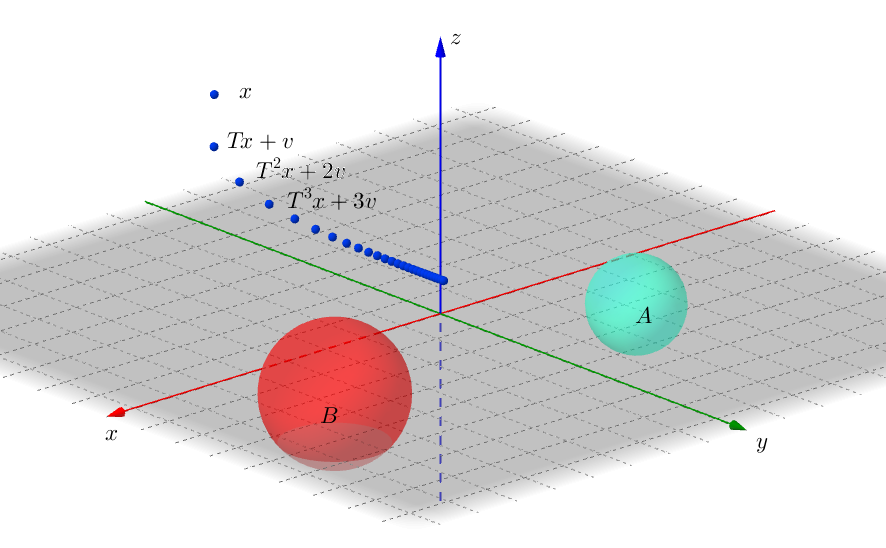}
\caption{A \texttt{GeoGebra} \cite{geogebra}
snapshot that illustrates \cref{ex:conj:R3}.
The first few terms of the 
 sequence $(T^n x+nv)_\nnn$ (blue points)
 are depicted. 
}
\label{fig_2_balls}
\end{center}
\end{figure}

\section*{Acknowledgments}
HHB was partially supported by the Natural Sciences and
Engineering Research Council of Canada 
and by the Canada Research Chair Program.
MND was partially supported by an NSERC accelerator grant of HHB.

\end{document}